\pdfoutput=1
\RequirePackage{ifpdf}
\ifpdf 
\documentclass[pdftex]{sigma}
\else
\documentclass{sigma}
\fi

\numberwithin{equation}{section}

\newtheorem{Theorem}{Theorem}[section]
\newtheorem*{Theorem*}{Theorem}
\newtheorem{Corollary}[Theorem]{Corollary}

\newtheorem{Proposition}[Theorem]{Proposition}

\theoremstyle{definition}
\newtheorem{Definition}[Theorem]{Definition}

\newtheorem{Example}[Theorem]{Example}
\newtheorem{Remark}[Theorem]{Remark}

\usepackage{bbm}
\usepackage{tikz-cd}
\usepackage{stmaryrd}
\usepackage{mathtools}

\DeclareMathOperator\Res{Res}
\DeclareMathOperator\id{id}
\newcommand\Fbinom[2]{\binom{#1}{#2}_{\!\!F}}
\def\lb{\llbracket}
\def\rb{\rrbracket}

\begin{document}

\allowdisplaybreaks

\newcommand{\arXivNumber}{2503.21390}

\renewcommand{\thefootnote}{}

\renewcommand{\PaperNumber}{004}

\FirstPageHeading

\ShortArticleName{Vertex $F$-Algebras and Their Associated Lie Algebra}

\ArticleName{Vertex $\boldsymbol{F}$-Algebras and Their Associated Lie Algebra\footnote{This paper is a~contribution to the Special Issue on Recent Advances in Vertex Operator Algebras in honor of James Lepowsky. The~full collection is available at \href{https://www.emis.de/journals/SIGMA/Lepowsky.html}{https://www.emis.de/journals/SIGMA/Lepowsky.html}}}

\Author{Markus UPMEIER}

\AuthorNameForHeading{M.~Upmeier}

\Address{Department of Mathematics, University of Aberdeen, Fraser Noble Building,\\
Elphinstone Rd, Aberdeen, AB24~3UE, UK}
\Email{\mail{markus.upmeier@abdn.ac.uk}}
\URLaddress{\url{https://www.abdn.ac.uk/people/Markus.Upmeier/}}

\ArticleDates{Received April 18, 2025, in final form January 03, 2026; Published online January 15, 2026}

\Abstract{Vertex $F$-algebras are a deformation of the concept of an ordinary vertex algebra in which the additive formal group law is replaced by an arbitrary formal group law $F$. The main theorem of this paper constructs a Lie algebra from a vertex $F$-algebra -- for the additive formal group law, this extends Borcherds' well-known construction for ordinary vertex algebras. Our construction involves the new concept of an $F$-residue and some other new algebraic concepts, which are deformations of familiar concepts for the special case of an additive formal group law.}


\Classification{17B69; 17B65}

\renewcommand{\thefootnote}{\arabic{footnote}}
\setcounter{footnote}{0}

\section{Introduction and results}
\label{s1}

Vertex algebras, popularized in mathematics by Borcherds \cite{Borc}, provide an algebraic framework for understanding two-dimensional conformal field theories and play a fundamental role in representation theory. A feature of vertex algebras is that their axioms implicitly rely on the additive formal group law $F(z,w)=z+w$. Vertex $F$-algebras, originally introduced by Li~\cite{Li}, generalize vertex algebras by replacing the additive formal group by an arbitrary formal group law $F(z,w)$.

A fundamental theorem of Borcherds \cite{Borc}, which has important applications to representation theory, states that every vertex algebra determines a Lie algebra. The main result of this paper extends this result and shows that every vertex $F$-algebra also determines a Lie algebra. However, the usual construction of the Lie bracket in the vertex algebra setting, via the residue of the state-to-field correspondence, does not directly apply to vertex $F$-algebras. To overcome this difficulty, we introduce a deformation of the concept of a residue, the $F$-residue $\Res_{z=0}^F$. Establishing the Lie algebra structure in the generalized setting requires significantly more technical work than in the classical case.

The following result is our main theorem, and it is proven in Section~\ref{s43}.

\begin{Theorem*}
Let $(V,\mathbbm{1},\mathcal{S},Y)$ be a vertex $F$-algebra \textup(see Definition~\textup{\ref{s3dfn1}}\textup). Then
\[
 [a,b] = \Res_{z=0}^F Y(a,z)b {\rm d}z
\]
defines a Lie bracket on the quotient $V/\sum_{n\geqslant1}\mathcal{S}^{(n)}(V)$.
\end{Theorem*}

Besides the potential applications to representation theory, there is further motivation for generalizing the concept of a vertex algebra stemming from enumerative geometry: recently, Joyce~\cite{Joy2,Joy} has shown that vertex algebras play a central role in enumerative geometry, where the Lie algebra associated to the ordinary homology $H_*(\mathcal{M})$ of a moduli stack $\mathcal{M}$ (for example, the stack of perfect complexes over a projective variety). The Lie bracket is used to formulate wall-crossing formulas of enumerative invariants. The study of enumerative invariants in generalized homology $E_*(\mathcal{M})$ is becoming increasingly popular (most notably, the theory of $K$-theoretic enumerative invariants). For complex oriented generalized homology theories~$E_*$ with formal group law $F$, similar constructions show that $E_*(\mathcal{M})$ is naturally a vertex $F$-algebra. Indeed, in~\cite{GroUp} the author and Gross proved that the generalized homology of an $H$-space (plus some extra data) naturally has a vertex $F$-algebra structure. This extends Joyce's construction to generalized homology. While a comprehensive theory of wall-crossing has not yet been developed for generalized enumerative invariants (for progress in the case of $K$-theory, see Liu~\cite{Liu}), it may well be expected that the wall-crossing formula will use the Lie bracket constructed in this paper.

The proof of our main theorem, as well as a better understanding of vertex $F$-algebras, requires the development of various deformations of familiar concepts ($F$-binomial coefficients, $F$-delta distributions, $F$-residues, and $F$-hyperderivatives), developed in Section~\ref{s3}. These constructions, which appear to be new to the literature, provide the tools for generalizing results for ordinary vertex algebras to vertex $F$-algebras. Section~\ref{s2} reviews some background of formal groups and establishes certain properties we will use later. In Section~\ref{s4}, we define vertex $F$-algebras (our definition is equivalent to that of Li~\cite{Li}), prove some meromorphicity properties, and prove our main theorem.

We use the following notation.
\begin{itemize}\itemsep=0pt
\item
$R$ a commutative ring with unit,
\item
$R\big[z^{\pm1}\big]$ the ring of Laurent polynomials,
\item
$R\lb z\rb$ the ring of formal power series $\sum_{i=0}^\infty a_i z^i$,
\item
$R\big\lb z^{\pm1}\big\rb$ the abelian group of bilateral Laurent series $\sum_{i=-\infty}^{+\infty} a_i z^i$ (note that $R\big\lb z^{\pm1}\big\rb$ is \emph{not} a ring, but only has a partially defined product),
\item
$R(\!( z)\!)$ the ring of meromorphic series having a pole at $0$.
\end{itemize}

\section{Background on formal groups}
\label{s2}

All rings $R$ are assumed to be commutative and unital. A \emph{formal group law} over $R$ is a formal power series $F(z,w)\in R\lb z,w\rb$ satisfying
\begin{align}
 F(z,w)=F(w,z)=z+w+O(zw),
\qquad
F(z,F(w,v))=F(F(z,w),v).\label{s1eqn1}
\end{align}
There is a unique \emph{inverse} $\iota(z)=-z+O\bigl(z^2\bigr)\in R\lb z\rb$ such that $F(z,\iota z)=0$.

\begin{Remark}
Our definition includes a choice of coordinate $z$ for the formal group.
\end{Remark}

\begin{Example}
Over $R=\mathbb{Z}[s]$, we have the formal group law
\begin{equation}\label{OneParameterFGL}
 F_s(z,w)=z+w+s\cdot zw.
\end{equation}
We view $F_s$ as a family of group laws over the affine line. At $s=0$, it specializes to the \emph{additive group law} $F_a$ and at $s=1$ to the \emph{multiplicative group law} $F_m$. All polynomial formal group laws are of the form \eqref{OneParameterFGL}, and other examples must involve infinite series.
\end{Example}

\begin{Example}
Over the ring of modular forms $R=\mathbb{Z}\big[\frac12,\epsilon,\delta\big]$, we have the \emph{elliptic formal group~law}
\[
 F_{ell}(z,w)=\frac{z\sqrt{S(w)}+w\sqrt{S(z)}}{1-\epsilon z^2w^2},\qquad S(z)=1-2\delta z^2+\epsilon z^4.
\]
For $\epsilon=\delta^2$ it specializes to the group law
\smash{$
 \frac{z+w}{1+\delta zw}
$}
(for $\delta=1$ equivalent to $L$-genus) and for $\epsilon=0$ it specializes to
\[
 z\sqrt{1-2\delta w^2}+w\sqrt{1-2\delta z^2}
\]
(for $\delta=-1/2$ equivalent to $\hat{A}$-genus).
\end{Example}

\begin{Example}
\label{PrimeFGL}
 Let $p$ be a prime number and fix a positive power $q=p^h$. Set $\phi_p(z)=z+\sum_{n\geqslant1}p^{-n}z^{q^n}$. In \cite{Haze}, it is shown that
\[
 F_p(z,w)=\phi_p^{-1}(\phi_p(z)+\phi_p(w))
\]
 has integer coefficients and thus defines a formal group law over $\mathbb{Z}$.
\end{Example}

Using the notation $F^{m,n}=\frac{\partial^{m+n}}{\partial z^m\partial w^n}F$ for derivatives, the axioms imply
\begin{align*}
 F^{m,n}(z,w)=F^{n,m}(w,z),\qquad
 F^{m,0}(z,0)=F^{0,m}(0,z)=\delta_{m,1}.
\end{align*}

The associativity law is the most interesting axiom. By differentiating it, one obtains many complicated identities, for example
\begin{equation}
\label{ComplicatedIdentity}
F^{0,1}(z,w)F^{1,0}(0,w)=F^{1,0}(z,w)F^{0,1}(z,0).
\end{equation}
Identities of this kind will be important later, so we now introduce a systematic method for proving these based on Lazard's theorem. We first discuss the invariant $1$-form and the logarithm of a formal group law.

The \emph{invariant differential} of a formal group law $F(z,w)$ is the unique $1$-form
\begin{equation}\label{Invariant_One_Form}
 \theta_F=p_F(z){\rm d}z,\qquad p_F(z)\in R\lb z\rb,
\end{equation}
satisfying
$
 p_F(0)=1$, $ F^*(\theta_F)=\pi_1^*(\theta_F)+\pi_2^*(\theta_F)$.
Equivalently,
\begin{equation}
\label{Invariant_One_Form_Equivalent_Char}
 p_F(F(z,w))\bigl(F^{1,0}(z,w){\rm d}z+F^{0,1}(z,w){\rm d}w\bigr)=p_F(z){\rm d}z+p_F(w){\rm d}w.
\end{equation}
Using this characterization of $\theta_F$ one verifies
\begin{equation}
\label{Invariant_Form_Inverse}
 -\iota^*(\theta_F)=\theta_F.
\end{equation}
One checks $p_F(z)=F^{0,1}(z,0)^{-1}$, which is well-defined as $F^{0,1}(z,0)=1+\cdots$.

Suppose now that $\mathbb{Q}\subset R$. Then every formal group law has a unique (formal) \emph{logarithm} $\phi\in R\lb z\rb$ satisfying
\begin{align}
\label{s1eqn8}
 \phi(F(z,w))=\phi(z)+\phi(w), \qquad\phi(0)=0, \qquad\phi'(0)=1.
\end{align}
Indeed, taking $\partial/\partial w$ of \eqref{s1eqn8} and setting $w=0$ we find that $\phi'(z){\rm d}z=\theta_F$ is the invariant $1$-form which, since $\mathbb{Q}\subset R$, has a primitive $\theta_F=d\phi$. For this reason, we write $p_F(z)=\phi'(z)$ with the caveat that $\phi(z)$ is only defined if $\mathbb{Q}\subset R$.
The composition inverse $\phi^{-1}(x)\in R\lb x\rb$ is called the \emph{exponential}.

\begin{Example}
For $s\neq0$, the logarithm of $F_s$ is $s^{-1}\log(1+sz)$. For $F_{ell}$, the logarithm is the functional inverse of the indefinite elliptic integral $\int S(z)^{-1/2}{\rm d}z$.
\end{Example}

\begin{Example}
 The logarithm in Example~\ref{PrimeFGL} is the series $\phi(z)$ defined there. For $p=q$, we have
 \[
 p_F(z)^{-1}=F^{0,1}(z,0)=\frac{1}{\phi'(z)}=\prod_{k=1}^\infty\frac{1}{1+z^{p^k}}.
 \]
 Hence the coefficient of $z^d$ in $p_F(z)^{-1}$ is the sum of $(-1)^k$ over all partitions $p^{n_1}+\dots+p^{n_k}=d$ into positive powers of the prime~$p$.
\end{Example}

Using variables as coefficients defines the \emph{universal formal group law}
\[
F_L=\sum_{n,m\geqslant0}a_{n,m}z^nw^m
\]
over the \emph{Lazard ring} $L$, the quotient of the polynomials $\mathbb{Z}[a_{n,m}]$ modulo all the relations contained in \eqref{s1eqn1}. By construction, every formal group law $F$ is obtained from $F_L$ by reduction of coefficients $F=u_*(F_L)$ along a unique ring homomorphism $u\colon L\to R$.

\begin{Theorem}[{Lazard~\cite{Laz}}]
$L\cong\mathbb{Z}[p_1,p_2,\dots]$ is a polynomial ring over the integers. In particular, $L$ is torsion-free.
\end{Theorem}

This implies that it suffices to prove statements about formal group laws over the rationals, where one may restrict to laws of the form
\begin{equation}\label{s1eqn9}
 F(z,w)=\phi^{-1}(\phi(z)+\phi(w)).
\end{equation}
Using this method, \eqref{ComplicatedIdentity} has a simple chain rule proof. Moreover, the following proposition would be very difficult to prove without this new method.

\begin{Proposition}
 The series $G(z,w)\in R\lb z,w\rb$ defined by
\begin{equation}\label{differByUnit}
F(z,\iota w)=G(z,w)\cdot(z-w)
\end{equation}
is a unit and converges on the diagonal to
\begin{equation}
\label{Gdiagonal}
 G(z,z)=\phi'(z).
\end{equation}
\end{Proposition}

\begin{proof}
 The coefficients of $G(z,w)$ in \eqref{differByUnit} satisfy a recursion which is easily solved inductively. The main point is to prove \eqref{Gdiagonal} for the group law \eqref{s1eqn9}. Set
 \[
 \psi(x)=\frac{\phi^{-1}(x)}{x}=1+O(x).
 \]
 Then
 \[
 G(z,w)=\frac{\phi^{-1}(\phi(z)-\phi(w))}{z-w}=\psi(\phi(z)-\phi(w))\cdot \frac{\phi(z)-\phi(w)}{z-w}.
 \]
 Substituting $w=z$, the first factor is $\psi(0)=1$ and the second is $\phi'(z)$.
\end{proof}

\section[Formal calculus and F-residues]{Formal calculus and $\boldsymbol{F}$-residues}
\label{s3}

\subsection{Bilateral and Laurent series}
\label{s31}

Let $R\big\lb z_1^{\pm1},\dots,z_n^{\pm1}\big\rb$ be the space of \emph{bilateral Laurent series}
\begin{equation}
\label{bilateral}
 f=\sum_{i_1,\dots,i_n\in\mathbb{Z}} a_{i_1,\dots,i_n} z_1^{i_1}\cdots z_n^{i_n}
\end{equation}
in variables $z_1, \dots, z_n$. This is an abelian group under addition, but the product is only partially defined. The product of \eqref{bilateral} with
\[
 g=\sum_{j_1,\dots,j_n\in\mathbb{Z}} b_{j_1,\dots,j_n} z_1^{j_1}\cdots z_n^{j_n}
\]
is said to \emph{converge} if each of the coefficients in
\[
 fg = \sum_{k_1,\dots,k_n\in\mathbb{Z}} \Biggl(\sum_{i_1+j_1=k_1,\dots, i_n+j_n=k_n}a_{i_1,\dots,i_n}b_{j_1,\dots,j_n}\Biggr)z_1^{k_1}\cdots z_n^{k_n}
\]
reduces to a sum with only finitely many non-zero terms. By \cite[p.\ 24]{LL}, this product is associative if all products $fg$, $gh$, $(fg)h$, $f(gh)$ converge and the triple product converges.

Similarly, we say that $f(z,w)=\sum_{i,j\in\mathbb{Z}} a_{i,j}w^iz^j\in R\lb z,w\rb$ \emph{converges on the diagonal} if each of the coefficients in
\[
 f(z,z)=\sum_{n\in\mathbb{Z}} \Biggl(\sum_{i+j=n} a_{i,j}\Biggr)z^n
\]
reduces to a sum with only finitely many non-zero terms.

We say that \eqref{bilateral} is a \emph{formal Laurent series} if $a_{i_1,\dots, i_n}=0$ for all but finitely many negative indices. The subspace of formal Laurent series
$
 R(\!( z_1,\dots,z_n)\!)
$
is a ring since all products converge. For $f\in R(\!( z_1,\dots,z_n)\!) $, we say also that $f$ is \emph{meromorphic} in the variables $z_1,\dots, z_n$. We have the subring
$
 R\lb \underline{z}_1,\dots,\underline{z}_n\rb
$
of \emph{formal power series} which are also said to be \emph{holomorphic} in $\underline{z}_1,\dots, \underline{z}_n$. It is useful to underline holomorphic variables in some contexts below.

\begin{Remark}
This generalizes to series with coefficients in an $R$-module $M$. Then $R(\!( z_1,\dots, z_n)\!)$ acts on $M(\!( z_1,\dots, z_n)\!)$. We leave this extension to the reader.
\end{Remark}

\begin{Proposition}
\label{Laurent-invbar}
 Let $f=\sum_{n\geqslant N} a_nz^n\in R(\!( z)\!) $ with lowest coefficient $a_N\neq0$. Then $f$ is invertible in $R(\!( z)\!) $ if and only if $a_N$ is invertible in~$R$, with inverse in $z^N R\lb z\rb$. In particular, all integer powers $f^n\in R(\!( z)\!) $ are defined in this case.
\end{Proposition}

\begin{proof}
 Suppose that $a_N$ is invertible and factor $f=a_Nz^N(1+zg)$ for $g\in R\lb z\rb$. Formally applying Newton's binomial theorem, we define
 \begin{equation}
 \label{eqn:newton-inverse}
 f^n=a_N^nz^{nN}(1+zg)^n=a_N^n\sum_{k=0}^\infty\binom{n}{k}z^{nN+k} g^k.
 \end{equation}
 Clearly, $f^n\cdot f^m = f^{n+m}$ by the binomial identity $\sum_{j=k+\ell}\binom{n}{k}\binom{m}{\ell}=\binom{n+m}{j}$ and $f^n\in z^{nN}R\lb z\rb$. Putting $n=1$, $m=-1$, shows that $f(z)$ is invertible.
\end{proof}

\subsection{Expansions}
\label{s32}

Expansion maps play a central role in modern formulations of vertex algebras, so we briefly review them here.

\begin{Definition}
Let $(x_1, \dots, x_m), \dots, (z_1, \dots, z_n)$ be tuples of formal variables, where singleton brackets will be dropped from the notation. We view the iterate Laurent series ring as a subset of the space of bilateral series:
\[
 R(\!( x_1,\dots, x_m)\!)\cdots(\!( z_1,\dots, z_n)\!)\subset R\big\lb x_1^{\pm1},\dots,x_m^{\pm1},\dots, z_1^{\pm1},\dots,z_n^{\pm1}\big\rb.
\]
Let $j_{(x_1,\dots, x_m),\dots, (z_1,\dots, z_n)}$ be the natural inclusion of $R(\!( x_1,\dots,x_m,\dots,z_1,\dots,z_n)\!)$, the un-iterated Laurent ring, into the ring $R(\!( x_1,\dots, x_m)\!)\cdots(\!( z_1,\dots, z_n)\!)$. Moreover, let $S_{(x_1,\dots, x_m),\dots, (z_1,\dots, z_n)}\subset R(\!( x_1,\dots,x_m,\dots,z_1,\dots,z_n)\!)$ be the multiplicative set consisting of those Laurent series whose image under $j_{(x_1,\dots, x_m),\dots, (z_1,\dots, z_n)}$ is invertible in the iterated Laurent ring.

The \emph{expansion map} $i_{(x_1,\dots, x_m),\dots, (z_1,\dots, z_n)}$ is the localization of the map $j_{(x_1,\dots, x_m),\dots, (z_1,\dots, z_n)}$ at~${S_{(x_1,\dots, x_m),\dots, (z_1,\dots, z_n)}}$ shown in the diagram
\[
\begin{tikzcd}[column sep=-7ex]
\hskip-3cm R(\!( x_1,\dots,x_m,\dots,z_1,\dots,z_n)\!)\dar\arrow[r,"j_{(x_1,\dots, x_m),\dots, (z_1,\dots, z_n)}" yshift=1.5ex]& R(\!( x_1,\dots, x_m)\!)\cdots(\!( z_1,\dots, z_n)\!)\\
S_{(x_1,\dots, x_m),\dots, (z_1,\dots, z_n)}^{-1}R(\!( x_1,\dots,x_m,\dots,z_1,\dots,z_n)\!).\arrow[ru,dashed,"i_{(x_1,\dots, x_m),\dots, (z_1,\dots, z_n)}"' xshift=1ex]
\end{tikzcd}
\]
\end{Definition}

\begin{Example}
\label{Ex_Main_Example_Expansions}
Let $F(z,w)$ be a formal group law. The images of the series $F(z,w)$, $F(z,\iota w)$, $F(\iota z,w)$, $\iota F(z,w)$ in $R(\!( z)\!)(\!( w)\!)$ are invertible because their lowest coefficients are the units $z,\iota(z)$ in $R(\!( z)\!)$. Therefore, $F(z,w)$, $F(z,\iota w)$, $F(\iota z,w)$, $\iota F(z,w)$ are in $S_{z,w}$. Hence there are well-defined integer powers
\[
 i_{z,w}F(z,w)^n,\qquad i_{z,w}F(\iota z,w)^n,\qquad i_{z,w}F(z,\iota w)^n,\qquad i_{z,w}(\iota F(z,w))^n,\qquad \forall n\in\mathbb{Z},
\]
which are elements of $R(\!( z)\!)(\!( w)\!)\subset R\big\lb z^{\pm1}, w^{\pm 1}\big\rb$. These are computed by first viewing $F(z,w)$, $F(z,\iota w)$, $F(\iota z,w)$, $\iota F(z,w)$ as elements of the iterated ring $R(\!( z)\!)(\!( w)\!)$ and then forming the $n$-th power there.
\end{Example}

\begin{Remark}
For ordinary vertex algebras, the expansion maps $R(\!( z,w)\!)[z-w]^{-1}\to R(\!( z)\!)(\!( w)\!)$ are defined on the localization by a single element $z-w$. For vertex $F$-algebras, it becomes necessary to localize $F(z,w)$, $F(z,\iota w)$, $F(\iota z,w)$, $\iota F(z,w)$. In the case of several variables, even more complicated expressions in $F$ must be localized (for example, in the proof of Proposition~\ref{PropJacobiDelta}). Since these expressions are difficult to list systematically, we define the expansion maps on the universal localization, for example, $i_{z,w}\colon S_{z,w}^{-1}R(\!( z,w)\!)\to R(\!( z)\!)(\!( w)\!)$.
\end{Remark}

We can also include holomorphic variables which we indicate by an underline. Note that expansion maps preserve products,
 \[
 i_{(x_1,\dots, x_m),\dots, (z_1,\dots, z_n)}(fg)
=i_{(x_1,\dots, x_m),\dots, (z_1,\dots, z_n)}(f)
\cdot i_{(x_1,\dots, x_m),\dots, (z_1,\dots, z_n)}(g).
 \]

\begin{Remark}
The grouping of variables is important, because the algebra in $R(\!( z_1)\!)(\!( z_2)\!)$ and~${R(\!( z_2)\!)(\!( z_1)\!)}$ is different. For example, the elements $i_{z_1,z_2}(z_1\pm z_2)^n$ and $i_{z_2,z_1}(z_1\pm z_2)^n$ are different in $R\big\lb z_1^{\pm1},z_2^{\pm1}\big\rb$ when $n<0$. Expansion maps specify the ring in which the algebraic operations are performed. For holomorphic variables, the grouping is unimportant since
\[
 R\lb\underline z_1\rb \lb\underline z_2\rb = R\lb\underline z_2\rb \lb\underline z_1\rb = R\lb\underline z_1,\underline z_2\rb.
\]
Hence $i_{\underline z_1,\underline z_2}=i_{\underline z_2,\underline z_1}$.
\end{Remark}

Note that $R(\!( z_1,z_2)\!)=R(\!( z_1)\!)(\!( z_2)\!)\cap R(\!( z_2)\!)(\!( z_1)\!)$ is a proper intersection since, for example, $\sum_{n=0}^\infty z_1^{-n}z_2^n\in R(\!( z_1)\!)(\!( z_2)\!)\setminus R(\!( z_1,z_2)\!)$. On the intersection, we have
\begin{equation}\label{swap-expansion}
 i_{z_1,z_2}f=i_{z_2,z_1}f,\qquad\forall f\in R(\!( z_1,z_2)\!).
\end{equation}
In particular, $R\big\lb z_1^{\pm1},z_2^{\pm1}\big\rb$ is an $R(\!( z_1,z_2)\!)$-module and expansions are linear,
\begin{equation}
\label{PowerSeriesLinearity}
 i_{z_1,z_2}(f\cdot g)=f\cdot i_{z_1,z_2}(g),\qquad \forall f\in R(\!( z_1,z_2)\!), \ g\in R\big\lb z_1^{\pm1},z_2^{\pm1}\big\rb.
\end{equation}

\begin{Remark}
 Traditionally, the expansion maps are combined with \eqref{eqn:newton-inverse} and $i_{z_1,z_2}(z_1\pm z_2)^n$ is defined as an explicit series. This places the emphasis on properties of binomial coefficients, which are actually just laws of exponentiation.
\end{Remark}

\subsection{Filtrations}\label{s33}

 To determine which substitutions are well-defined, we will introduce topologies on iterate Laurent rings. These will all be induced by filtrations.

\begin{Definition}
A \emph{filtration} on a ring $A$ is a decreasing sequence of subgroups
\begin{align}\label{filtration}
&\cdots\supset A_{(n)}\supset A_{(n+1)}\supset\cdots,\qquad n\in\mathbb{Z},
\end{align}
such that $1\in A_{(0)}$ and $A_{(n)}\cdot A_{(m)}\subset A_{(n+m)}$. A filtration is said to be \emph{exhaustive} if $\bigcup A_{(n)}=A$ and \emph{separated} if $\bigcap A_{(n)}=\{0\}$. We will always assume that filtrations are separated.
\end{Definition}

 There is a unique Hausdorff topology on a filtered ring for which addition and multiplication are continuous and for which \eqref{filtration} is a fundamental system of neighborhoods of zero. This topology is induced by the uniform structure given by the pseudometric
\[
 d(x,y)=
 \begin{cases}
 0 & {\iff} x-y\in\bigcap_{n\in\mathbb{Z}} A_{(n)},\\
 2^{-n} & {\iff} x-y \in A_{(n)}\setminus A_{(n+1)},\\
 +\infty & {\iff} x-y \notin\bigcup_{n\in\mathbb{Z}} A_{(n)}.
 \end{cases}
\]
A filtered ring is said to be \emph{complete} if it is complete as a uniform space. A \emph{morphism} of filtered rings is a ring homomorphism $f\colon A\to B$ such that $f(A_{(n)})\subset B_{(n)}$ for all $n$. Filtered maps are uniformly continuous.

\begin{Example}
 Let $\mathfrak{a}\triangleleft A$ be an ideal. The \emph{$\mathfrak{a}$-adic filtration} on $A$ is defined by $A_{(n)}=\mathfrak{a}^n$ for~${n>0}$ and $A_{(n)}=A$ for $n\leqslant0$. If $\mathfrak{a}=(0)$, this is the \emph{trivial filtration} on $A$, which induces the discrete topology.

 For example, $A\lb z\rb$ carries the $(z)$-adic filtration where $A\lb z\rb_{(n)}=z^nA\lb z\rb$ for $n\geqslant0$. We also call the filtration on $A(\!( z)\!)$ defined by $A(\!( z)\!)_{(n)}=z^nA\lb z\rb$ for $n\in\mathbb{Z}$ the \emph{$(z)$-adic filtration}, although $(z)$ is not an ideal in $A(\!( z)\!)$.
\end{Example}

\subsection{Substitutions}\label{s34}

The following is the main technical result for defining substitutions. Parts (a)--(c) extend \cite[Chapter~III, Section~2.6]{Bour} to the case where $A_{(0)}\neq A$.

\begin{Proposition}\quad
\begin{itemize}\itemsep=0pt
\item[$(a)$]
 Let $A$ be a filtered ring. Then
 \begin{equation}
 \label{filtration2}
 A(\!( z)\!)_{(n)}=\Bigl\{\sum_{i\in\mathbb{Z}} a_iz^i \,\Big|\, a_i\in A_{(n-i)}\Bigr\},\qquad n\in\mathbb{Z},
 \end{equation}
 defines a filtration on the Laurent ring $A(\!( z)\!)$. If $A$ is separated, then so is $A(\!( z)\!)$ and the power series ring $A\lb z\rb$ is a closed subspace with filtration
 \begin{equation}
 \label{filtration3}
 A\lb z\rb_{(n)}=A(\!( z)\!)_{(n)}\cap A\lb z\rb.
 \end{equation}
\item[$(2)$]
 If $A$ is complete, then so are $A(\!( z)\!)$ and $A\lb z\rb$.
\item[$(3)$]
 If $A_{(0)}=A$, then the polynomials $A[z]$ are dense in $A\lb z\rb$ and the Laurent polynomials $A\big[z^{\pm1}\big]$ are dense in $A(\!( z)\!)$.
\item[$(4)$]
\textup(Universal properties\textup) For each filtered morphism $f\colon A\to B$ into a complete ring $B$ and each invertible element $b\in B_{(1)}$, there is a unique filtered morphism $\bar{f}\colon A(\!( z)\!)\to B$ with~${\bar{f}|_{A}=f}$ and $\bar{f}(z)=b$. Here $A(\!( z)\!)$ is given the filtration generated by~\eqref{filtration2} and the $(z)$-adic filtration.
Pictorially,
 \begin{equation}
 \label{EvalCompletionLocalize}
 \begin{tikzcd}
 A(\!( z)\!)\arrow[r,dashed,"{z\mapsto b}"]& B.\\
 A\uar\arrow[ru,"f"']
 \end{tikzcd}
 \end{equation}
 Similarly for multivariable Laurent rings $A(\!( z_1,\dots, z_n)\!)$ and also for power series rings, where we drop the invertibility condition.
\end{itemize}
\end{Proposition}

\begin{proof}
 (a)~It is easy to check that \eqref{filtration2} defines a filtration of $A(\!( z)\!)$, which is separated if the original filtration is so. Moreover, $A\lb z\rb$ is a closed subspace.

 (b)~To prove that $A(\!( z)\!)$ is complete, let $f_n(z)=\sum_{i\in\mathbb{Z}} a_{i,n}z^i$, $n\in\mathbb{N}$, be a Cauchy sequence. For each $N$, there is $n_0$ such that for all $n,m\geqslant n_0$ and $i$ we have
 \begin{equation} \label{Cauchy}
 \forall n,m\geqslant n_0\colon \ a_{i,n}-a_{i,m}\in A_{(N-i)}.
 \end{equation}
 In particular, for each fixed $i$ the sequence $(a_{i,n})_{n\in\mathbb{N}}$ is Cauchy in $A$, hence converges to some~${a_i\in A}$. We may then take $m\to\infty$ in \eqref{Cauchy} and get
 \[
 \forall n\geqslant n_0\colon\ a_{i,n}-a_i\in A_{(N-i)}.
 \]
 This proves that $f_n(z)\to \sum_{i\in\mathbb{Z}} a_iz^i$ as $n\to\infty$.

 (c)~If $A_{(0)}=A$, then we have $az^i\in A\lb z\rb_{(i)}$ for any $a\in A$. Hence the truncation $\sum_{i=0}^{N-1} a_iz^i\in A[z]$ is in the $A_{(N)}$-neighborhood of $\sum_{i=0}^\infty a_iz^i$, showing that the polynomials are dense. Similarly for Laurent polynomials.

 (d)~We prove the universal property for $A\lb z\rb$. By the universal property of polynomial rings, $(f,b)$ uniquely induce a ring homomorphism
 \begin{equation}
 \label{poly-map}
 A[z]\longrightarrow B,\qquad z\longmapsto b.
 \end{equation}
 First, put the trivial filtration on $A$. Then \eqref{poly-map} is filtered for the $(z)$-adic topology since ${b\in B_{(1)}}$. Since $B$ is complete, we may extend to a filtered map $\bar{f}$ on the completion $A\lb z\rb$. Observe that $\bar{f}$ is also automatically filtered for \eqref{filtration2}. Hence $\bar{f}$ is filtered for the generated filtration. Conversely, if $\bar{f}$ is filtered for the generated filtration, then it is also filtered for the $(z)$-adic filtration for which $A[z]$ is dense. This proves uniqueness. Similarly for Laurent series.\looseness=-1
\end{proof}

\begin{Remark}
 If $A_{(0)}\neq A$, the polynomials $A[x]$ are not dense in $A\lb x\rb$ and the filtration in the universal property differs from \eqref{filtration3}.
\end{Remark}

Applying \eqref{filtration2} inductively starting with the trivial filtration on $R$, we obtain the \emph{complete filtration} on iterate Laurent rings $R(\!( x_1)\!)\cdots(\!( x_m)\!)$, where to simplify notation we work in the single variable case. For $h(x_1,\dots, x_m)$, to be in positive complete filtration means that each monomial occurring in $h$ has strictly positive total degree. We also have the \emph{generated filtration} on $R(\!( x_1)\!)\cdots(\!( x_m)\!)$, obtained by using the generated filtration inductively at each step.

\begin{Corollary}
Let $f_i(y_1,\dots,y_n)\in R(\!( y_1)\!)\cdots(\!( y_n)\!)$, $1\leqslant i\leqslant m$, be invertible and in positive complete filtration. There is a unique \emph{substitution morphism}{\samepage
\[
 |_{x_i^{\pm1}\to f_i(y_1,\dots,y_n)^{\pm1}}\colon\ R(\!( x_1)\!)\cdots(\!( x_m)\!)\longrightarrow R(\!( y_1)\!)\cdots(\!( y_n)\!), x_i\longmapsto f_i(y_1,\dots,y_n),
\]
taking the generated filtration to the complete filtration.}

Moreover, given $g_j(z_1,\dots, z_p)\in R(\!( z_1)\!)\cdots(\!( z_p)\!)$, $1\leqslant j\leqslant n$, invertible and in positive complete filtration we have \emph{functoriality}
\[
 |_{y_j^{\pm1}\to g_i(z_1,\dots,y_p)^{\pm1}} \circ |_{x_i^{\pm1}\to f_i(y_1,\dots,y_n)^{\pm1}}=|_{x_i^{\pm1}\to h_i(z_1,\dots, z_p)^{\pm1}},
\]
where $h_i(z_1,\dots, z_p)=f_i(y_1,\dots,y_n)|_{y_j^{\pm1}\to g_i(z_1,\dots,y_p)^{\pm1}}$.
\end{Corollary}

For example, there are mutually inverse bijections
\begin{gather}
 R(\!( v)\!)(\!( w)\!)\longrightarrow R(\!( z)\!)(\!( w)\!),\qquad v\longmapsto i_{z,w}F(z,w), \label{substitution1}\\
 R(\!( z)\!)(\!( w)\!)\longrightarrow R(\!( v)\!)(\!( w)\!),\qquad z\longmapsto i_{v,w}F(v,\iota(w)).\label{substitution3}
\end{gather}
Moreover, the universal property \eqref{EvalCompletionLocalize} yields a unique filtered morphism
 \begin{equation}
 \label{substitution2}
 R(\!( v)\!)\longrightarrow R(\!( z)\!)(\!( w)\!),\qquad v\longmapsto i_{z,w}F(z,w).
 \end{equation}

\begin{Remark}
 For vertex algebras, one may restrict to substitutions in Laurent polynomials. As a formal group law may involve infinitely many powers, this is no longer possible for vertex $F$-algebras.
\end{Remark}

\subsection[F-binomial coefficient]{$\boldsymbol{F}$-binomial coefficients}\label{s35}

Let $F(z,w)$ be a formal group law. As observed in Example~\ref{Ex_Main_Example_Expansions}, for all $n\in\mathbb{Z}$ there are well-defined integer powers
\[
 i_{z,w}F(z,w)^n,\qquad i_{z,w}F(\iota z,w)^n,\qquad i_{z,w}F(z,\iota w)^n,\qquad i_{z,w}(\iota F(z,w))^n.
\]

\begin{Definition}
The coefficients of the expansion $i_{z,w}F(z,w)^n$ are called the \emph{F-binomial coefficients}, so by definition
\begin{equation}
\label{def-F-binom}
 i_{z,w}F(z,w)^n=\sum_{i,j\in\mathbb{Z}}\Fbinom{n}{i,j} z^iw^j,\qquad\forall n\in\mathbb{Z}.
\end{equation}
\end{Definition}

\begin{Proposition}
The $F$-binomial coefficients satisfy the following identities:
\begin{gather*}
 \Fbinom{n}{i,j}=0
 \qquad\text{if} \ j<0 \ \text{or} \ i+j<n,\\
 \Fbinom{n}{i,0}=\delta_i^m,\qquad
\Fbinom{n}{i,j}=\Fbinom{n}{j,i}
\qquad\text{if}\ n\geqslant0,\\
 \Fbinom{m+n}{r,s}=\sum_{\substack{i+k=r\\j+\ell=s}} \Fbinom{m}{i,j}\Fbinom{n}{k,\ell}.
\end{gather*}
\end{Proposition}

\begin{Example}
For the one-parameter group law \eqref{OneParameterFGL}, we have
\[
 \binom{n}{i,j}_{\!\!F_s}=\binom{n}{j}\binom{j}{i+j-n}s^{i+j-n}.
\]
In particular, for the additive group law ($s=0$), we recover the ordinary binomial coefficients $\binom{n}{i,j}_{\!\!F_a}=\binom{n}{j}$ if $i+j=n$ and $\binom{n}{i,j}_{\!\!F_a}=0$ else. For the multiplicative group law, put $s=1$.
\end{Example}

\subsection[F-delta distributions]{$\boldsymbol{F}$-delta distributions}
\label{s36}


\begin{Definition}
 The group law \emph{$F$-delta distribution} is
 \[
 z^{-1}\delta_F\left(\frac{w}{z}\right) = i_{z,\underline w}F(z,\iota w)^{-1}-i_{w,\underline z}F(z,\iota w)^{-1}.
 \]
\end{Definition}

\begin{Example}
For the additive group law, we get the classical distribution
\begin{equation}
\label{ClassicalDelta}
 z^{-1}\delta_{F_a}\left(\frac{w}{z}\right)=i_{z,\underline w}(z-w)^{-1}-i_{w,\underline z}(z-w)^{-1}=\sum_{n\in\mathbb{Z}} w^nz^{-n-1}.
\end{equation}
More generally, for the one-parameter formal group law \eqref{OneParameterFGL}, we have
\[
 \delta_{F_s}\left(\frac{w}{z}\right)=(1+sw)\sum_{n\in\mathbb{Z}} w^nz^{-n}.
\]
\end{Example}

\begin{Proposition}\quad
\label{PropJacobiDelta}
\begin{itemize}\itemsep=0pt
\item[$(a)$]
 The F-delta distribution is supported on the diagonal,
 \begin{equation}
 \label{DiagonalSupport}
 \delta_F\left(\frac{w}{z}\right)f(z)=\delta_F\left(\frac{w}{z}\right)f(w),\qquad\forall f\in R(\!( z)\!).
 \end{equation}
 More generally, if $f(z,w)\in R(\!( z,w)\!)$ converges on the diagonal,
 \begin{equation}
 \label{DiagonalSupport2}
 \delta_F\left(\frac{w}{z}\right)f(z,w)=\delta_F\left(\frac{w}{z}\right)f(w,w).
 \end{equation}
 \item[$(2)$]
 We have the \emph{$F$-Jacobi identity},
 \begin{gather*}
 i_{z_1,\underline z_2}z_0^{-1}\delta_F\left(\frac{F(z_1,\iota z_2)}{z_0}\right)
 -i_{z_2,\underline z_1}z_0^{-1}\delta_F\left(\frac{F(z_1,\iota z_2)}{z_0}\right)
 =i_{z_1,\underline z_0}z_2^{-1}\delta_F\left(\frac{F(z_1,\iota z_0)}{z_2}\right).
 \end{gather*}
 \item[$(3)$]
 We have
 \[
 i_{z_1,\underline z_0}z_2^{-1}\delta_F\left(\frac{F(z_1,\iota z_0)}{z_2}\right)
 =i_{z_2,\underline z_0}z_1^{-1}\delta_F\left(\frac{F(z_0,z_2)}{z_1}\right).
 \]
 \end{itemize}
\end{Proposition}

\begin{proof}
(a)~The identities \eqref{DiagonalSupport} and \eqref{DiagonalSupport2} are well-known for the classical delta distribution. By \eqref{differByUnit} and \eqref{PowerSeriesLinearity}, the delta distributions are related by
\begin{equation}
\label{deFdeAdditive}
 \delta_F\left(\frac{w}{z}\right)=G(z,w)^{-1}\delta_{F_a}\left(\frac{w}{z}\right).
\end{equation}
Therefore, \eqref{DiagonalSupport} and \eqref{DiagonalSupport2} follow directly from the classical case.

(b)~Formally, the first term of the Jacobi identity is the substitution
\[
 i_{z_1,\underline z_2}z_0^{-1}\delta_F\left(\frac{F(z_1,\iota z_2)}{z_0}\right)
 =z_0^{-1}\delta_F\left(\frac{w}{z_0}\right)\Big|_{w^{\pm1}\to i_{z_1,\underline z_2}F(z_1,\iota \underline z_2)^{\pm1}}.
\]
Expanding the definition of $\delta_F$ and using functoriality of substitution, this gives
 \begin{align*}
 i_{z_1,\underline z_2}z_0^{-1}\delta_F\left(\frac{F(z_1,\iota z_2)}{z_0}\right)={}&
i_{z_0,\underline w}F(z_0,\iota\underline w)^{-1}\big|_{\underline w\to i_{\underline z_1, \underline z_2} F(\underline z_1,\iota \underline z_2)}\\
&-i_{w,\underline z_0}F(\underline z_0,\iota w)^{-1}\big|_{w^{\pm1}\to i_{z_1, \underline z_2} F(z_1,\iota \underline z_2)^{\pm1}}\\
={}&i_{z_0,(\underline z_1, \underline z_2)} F(z_0,\iota z_1, z_2)^{-1}-i_{z_1,(\underline z_2, \underline z_0)} F(z_0,\iota z_1, z_2)^{-1}.
 \end{align*}
In the same way,
\begin{align*}
 i_{z_2,\underline z_1}z_0^{-1}\delta_F\left(\frac{F(z_1,\iota z_2)}{z_0}\right)
 &=i_{z_0,(\underline z_2,\underline z_1)}F(z_0,\iota z_1, z_2)^{-1}-i_{z_2,(\underline z_1,\underline z_0)}F(z_0,\iota z_1, z_2)^{-1},\\
 i_{z_1,\underline z_0}z_2^{-1}\delta_F\left(\frac{F(z_1,\iota z_0)}{z_2}\right)
 &=i_{z_2,(\underline z_1, \underline z_0)}F(z_0,\iota z_1, z_2)^{-1}-i_{z_1,(\underline z_0, \underline z_2)}F(z_0,\iota z_1, z_2)^{-1}.
\end{align*}
Observe that by \eqref{swap-expansion} each term appears twice in the last three equations. Hence the terms cancel in pairs, which proves the $F$-Jacobi identity.

(c)~Notice that
\[
 i_{z_2,\underline z_0}z_1^{-1}\delta_F\left(\frac{F(z_0,z_2)}{z_1}\right)
=i_{z_1,(\underline z_0, \underline z_2)}F(z_1,\iota z_0,\iota z_2)^{-1} - i_{z_2,(\underline z_0, \underline z_1)}F(z_1,\iota z_0, \iota z_2)^{-1}
\]
can be multiplied by the holomorphic $\frac{F(\iota z_0,z_1,\iota z_2)}{F(z_0, \iota z_1, z_2)}$ which may be exchanged with the expansion. Hence
\begin{gather*}
 i_{z_2,\underline z_0}z_1^{-1}\delta_F\left(\frac{F(z_0,z_2)}{z_1}\right)\cdot\frac{F(\iota z_0,z_1,\iota z_2)}{F(z_0, \iota z_1, z_2)}\\
\qquad=i_{z_1,(\underline z_0, \underline z_2)}F(z_0,\iota z_1, z_2)^{-1} - i_{z_2,(\underline z_0, \underline z_1)}F(z_0,\iota z_1, z_2)^{-1}\\
\qquad\overset{\eqref{swap-expansion}}{=}-i_{z_1,\underline z_0}z_2^{-1}\delta_F\left(\frac{F(z_1,\iota z_0)}{z_2}\right).
\end{gather*}
Finally, using \eqref{DiagonalSupport} put $z_1=F(z_0,z_2)$ into $\frac{F(\iota z_0,z_1,\iota z_2)}{F(z_0, \iota z_1, z_2)}$ on the left to get $1$.
\end{proof}

By \eqref{DiagonalSupport2}, we may put $z=w$ in \eqref{deFdeAdditive} and use \eqref{Gdiagonal} to get
\begin{equation}
\label{deFdeAdditive2}
 \delta_F\left(\frac{w}{z}\right)\phi'(z)=\delta_{F_a}\left(\frac{w}{z}\right).
\end{equation}
Observe here that $\phi'(z)=F^{0,1}(z,0)^{-1}$ makes sense and \eqref{deFdeAdditive2} holds over an arbitrary ring $R$. There is a proof of the $F$-Jacobi identity based on \eqref{deFdeAdditive2} and the Jacobi identity for the additive group law, but we remind the reader that the proof of \eqref{Gdiagonal} used Lazard's theorem. We have given an elementary proof.

\subsection[F-residues]{$\boldsymbol{F}$-residues}
\label{s37}

Recall that the \emph{residue} of a formal Laurent series $f(z)=\sum a_nz^n\in R\big\lb z^{\pm1}\big\rb$ is defined as $\Res_{z=0} f(z){\rm d}z=a_{-1}$. We have the identities (see \cite[Chapter~13]{Rem})
\begin{gather}
\Res_{z=0} f^{(n)}(z){\rm d}z=0, \forall n\geqslant1,\qquad
\Res_{z=0} f'(z)g(z){\rm d}z=-\Res_{z=0} f(z)g'(z){\rm d}z,\nonumber\\
\Res_{w=0} f(h(w))h'(w){\rm d}w=\Res_{z=0} f(z){\rm d}z.\label{s2eqn3}
\end{gather}
Here $f,g\in R\big\lb z^{\pm1}\big\rb$ are formal Laurent series and $h$ is a (holomorphic) formal power series with~${h(0)=0}$ and $h'(0)\in R^\times$.\medskip

This terminology can be extended to formal group laws $F(z,w)$ as follows. Recall the invariant $1$-form $\theta_F$ from \eqref{Invariant_One_Form}. Then the \emph{F-residue} is defined as
$
 \Res_{z=0}^F f(z){\rm d}z=\Res_{z=0}f(z)\theta_F$.
Assuming the existence of a logarithm,
\begin{align}
\label{s2eqn5}
 \Res_{z=0}^F f(z){\rm d}z=\Res_{x=0} f(x){\rm d}x,
 \qquad\text{where}\quad f(x)=f\bigl(\phi^{-1}(x)\bigr).
\end{align}
From \eqref{Invariant_Form_Inverse}, we have
\begin{equation}
\label{residue-inversion}
 \Res_{z=0}^F f(\iota(z)){\rm d}z=-\Res_{z=0}^F f(z){\rm d}z.
\end{equation}
From \eqref{deFdeAdditive2} combined with \eqref{ClassicalDelta} and \eqref{DiagonalSupport}, we find
\[
 \Res_{z=0}^Fz^{-1}\delta_F\left(\frac{w}{z}\right){\rm d}z=\Res_{w=0}^Fz^{-1}\delta_F\left(\frac{w}{z}\right){\rm d}w=1.
\]

\subsection[F-hyperderivatives]{$\boldsymbol{F}$-hyperderivatives}
\label{s38}

Generalizing Hasse's number-theoretic notion to formal group laws, define the \emph{F-hyperderivative} of $f\in R(\!( z)\!)$, $\mathcal{S}^F_n f$, $n\geqslant0$, by expanding the substitution
\begin{equation}
\label{def-F-hyper}
 i_{z,\underline{w}}f(F(z,w))=\sum_{n\geqslant0}\bigl(\mathcal{S}^F_n f\bigr)(z)w^n.
\end{equation}
Hence $\mathcal{S}^F_nf(0)$ picks out the $n$-th coefficient of $f$, and we think of $\mathcal{S}^F_n f$ as a substitute for $(n!)^{-1}f^{(n)}(z)$. One easily shows
\begin{gather}
 \bigl(\mathcal{S}_0^F f\bigr)(z)=f(z),\qquad
 \bigl(\mathcal{S}_1^F f\bigr)(z)\theta_F=f'(z){\rm d}z,\qquad
 \mathcal{S}_n^F(\lambda f+g)=\lambda\mathcal{S}_n(f)+\mathcal{S}_n(g),\nonumber\\
 \mathcal{S}_n^F(f\cdot g)=\sum_{i+j=n} \mathcal{S}_i^F(f)\cdot \mathcal{S}_j^F(g),\label{s2eqn7}\\
 \mathcal{S}_m^F\mathcal{S}_n^F(f)=\mathcal{S}_n^F\mathcal{S}_m^F(f)=\sum_{k\geqslant0}\Fbinom{k}{m,n}\mathcal{S}_k^F(f).\label{s2eqn8}
\end{gather}
For the additive group law, \eqref{s2eqn8} implies $\bigl(\mathcal{S}_1^{F_a}\bigr)^{\circ n}=n!\cdot\mathcal{S}^{F_a}_n$, which shows the importance of considering higher hyperderivatives in the presence of torsion. Equation \eqref{s2eqn8} follows from the associativity of $F$.

Comparing \eqref{def-F-binom} and \eqref{def-F-hyper}, we find
\[
 \mathcal{S}_j^F(z^m)=\sum_{i\in\mathbb{Z}}\Fbinom{m}{i,j}z^i.
\]
From this, we see that \eqref{s2eqn8} implies
\[
 \sum_{i\in\mathbb{Z}}\Fbinom{m}{i,s}\Fbinom{i}{j,r}=\sum_{i\in\mathbb{Z}}\Fbinom{m}{i,r}\Fbinom{i}{j,s}=\sum_{k\geqslant0}\Fbinom{k}{r,s}\Fbinom{m}{j,i}
\]
for all $r,s,m,j\in\mathbb{Z}$.

\begin{Theorem}
\label{s2thm1}
For all $n\geqslant1$, we have
\begin{gather}
 \Res^F_{z=0}\mathcal{S}_n^Ff(z){\rm d}z=0,\label{s2eqn9}\\
 \Res^F_{z=0}\mathcal{S}_n^F f(z)\cdot g(z){\rm d}z=-\sum_{j=1}^n \Res_{z=0}^F\bigl[\mathcal{S}_{n-j}^Ff(z)\cdot \mathcal{S}_j^F g(z)\bigr]{\rm d}z.\label{s2eqn10}
\end{gather}
\end{Theorem}

\begin{proof}
 Summing \eqref{s2eqn9} times $w^n$ over $n$, we see by \eqref{def-F-hyper} that it suffices to show
 \[
 \Res^F_{z=0} i_{z,\underline{w}} f(F(z,w)){\rm d}z=\Res^F_{z=0}f(z){\rm d}z.
 \]
 We will apply the substitution rule \eqref{s2eqn3} to the power series $h(z)=F(z,w)-w$ with coefficients in the ring $R\lb w\rb$. This corresponds to viewing $w$ as a fixed constant. Putting ${\rm d}w=0$ in \eqref{Invariant_One_Form_Equivalent_Char} gives $p_F(F(z,w))F^{1,0}(z,w)=p_F(z)$. Clearly, $h(0)=0$ and $h'(0)=F^{1,0}(0,w)=p_F(w)^{-1}$ is a unit, so the hypotheses for \eqref{s2eqn3} are satisfied. Applying this and unravelling the definition gives
 \begin{align*}
 \Res^F_{z=0} i_{z,\underline{w}} f(F(z,w)){\rm d}z&=\Res_{z=0} i_{z,\underline{w}} f(F(z,w))p_F(z){\rm d}z\\
 &=\Res_{z=0} i_{z,\underline{w}} f(F(z,w))p_F(F(z,w))F^{1,0}(z,w){\rm d}z\\
 &=\Res_{z=0} i_{z,\underline{w}} f(h(z)+w)p_F(h(z)+w)h'(z){\rm d}z\\
 &=\Res_{v=0} i_{v,\underline{w}} f(v+w)p_F(v+w){\rm d}v.
 \end{align*}
 Let \[
 g(z)=f(z)p_F(z)=\sum_{n\in\mathbb{Z}}a_nz^n\in R(\!( z)\!).\]
 It remains to prove that
 \[
 \Res_{v=0} i_{v,\underline{w}}g(v+w){\rm d}v=\Res_{z=0}g(z){\rm d}z,
 \]
 which is a straight-forward computation
 \[
 \Res_{v=0} i_{v,\underline{w}}g(v+w){\rm d}v=\Res_{v=0} \sum_{n,k\in\mathbb{Z}}a_n\binom{n}{k}v^{n-k}w^k {\rm d}v=a_{-1},
 \]
 where we have used that $\binom{n}{n+1}=0$ for $n\neq-1$ and $\binom{-1}{0}=1$.

 Finally, combining \eqref{s2eqn7} with \eqref{s2eqn9} gives \eqref{s2eqn10}.
\end{proof}

\begin{Theorem}
 For all $f\in R\big\lb x_0^{\pm1}, x_1^{\pm1}, x_2^{\pm1}\big\rb$ such that the substitutions $f(x_1-x_2,x_1,x_2)$, $f(x_0,x_1,x_1-x_0)$ converge in the algebraic sense,
 \begin{gather}
 \Res_{x_1=0}\Res_{x_2=0}i_{x_1,x_2} f(x_1-x_2,x_1,x_2) {\rm d}x_2 {\rm d}x_1\nonumber\\
\qquad -\Res_{x_2=0}\Res_{x_1=0}i_{x_2,x_1} f(x_1-x_2,x_1,x_2) {\rm d}x_1 {\rm d}x_2\nonumber\\
 \phantom{\qquad -}{}=\Res_{x_1=0}\Res_{x_0=0}i_{x_1,x_0} f(x_0,x_1,x_1-x_0) {\rm d}x_0 {\rm d}x_1. \label{s2eqn11}
 \end{gather}
 More generally,
 \begin{gather}
 \Res^F_{z_1=0}\Res^F_{z_2=0}i_{z_1,z_2} f(F(z_1,\iota(z_2)),z_1,z_2) {\rm d}z_2 {\rm d}z_1\nonumber\\
 \qquad-\Res^F_{z_2=0}\Res^F_{z_1=0}i_{z_2,z_1} f(F(z_1,\iota(z_2)),z_1,z_2) {\rm d}z_1 {\rm d}z_2\nonumber\\
 \phantom{\qquad -}{}=\Res^F_{z_1=0}\Res^F_{z_0=0}i_{z_1,z_0} f(z_0,z_1,F(z_1,\iota(z_0))) {\rm d}z_0 {\rm d}z_1. \label{s2eqn12}
 \end{gather}
\end{Theorem}
\begin{proof}
 It suffices to check \eqref{s2eqn11} for $f=x_0^a x_1^b x_2^c$, where
 \begin{align*}
 i_{x_1,x_2} f(x_1-x_2,x_1,x_2)&=\sum_{k\geqslant0}(-1)^k\binom{a}{k} x_1^{a+b-k}x_2^{c+k},\\
 i_{x_2,x_1} f(x_1-x_2,x_1,x_2)&=\sum_{k\geqslant0}(-1)^{a+k}\binom{a}{k} x_1^{b+k}x_2^{a+c-k},\\
 i_{x_1,x_0} f(x_0,x_1,x_1-x_0)&=\sum_{k\geqslant0}(-1)^k\binom{c}{k}x_0^{a+k}x_1^{b+c-k}.
 \end{align*}
 Taking residues, we see that \eqref{s2eqn11} reduces to the binomial identity
 \[
 (-1)^{c-1}\binom{a}{1-c}-(-1)^{a+b-1}\binom{a}{1-b}=(-1)^{a-1}\binom{c}{1-a}
 \]
 for all $a+b+c+2=0$.

 It suffices to prove \eqref{s2eqn12} for the universal formal group law, where we may assume~\eqref{s1eqn9}. In this case, we can use \eqref{s2eqn5} to reduce \eqref{s2eqn12} to \eqref{s2eqn11} for the function $f(x_0,x_1,x_2)$, where $\phi(z_i)=x_i$, $i=0,1,2$.
\end{proof}

\section[Vertex F-algebras and their Lie algebras]{Vertex $\boldsymbol{F}$-algebras and their Lie algebras}
\label{s4}

\subsection{Axioms}
\label{s41}

\begin{Definition}
\label{s3dfn1}
A \emph{vertex $F$-algebra} over a formal group law $F(z,w)\in R\lb z,w\rb$ consists of data $(V,\mathbbm{1},\mathcal{S},Y)$ as follows:
\begin{itemize}\itemsep=0pt
\item
an $R$-module $V$ of \emph{states},
\item
a \emph{vacuum vector} $\mathbbm{1}\in V$,
\item
an $R$-linear \emph{$F$-shift operator}
\begin{align}
\label{F-shift}
\mathcal{S}(z)\colon V \longrightarrow V\lb z\rb,\qquad
\mathcal{S}(z)a=\sum_{n\geqslant0} \mathcal{S}^{(n)}(a)z^n,
\end{align}
\item
an $R$-linear \emph{state-to-field correspondence}
\begin{gather}
 V\otimes_R V \longrightarrow V(\!( z)\!),\qquad a\otimes b\longmapsto Y(a,z)b,\nonumber\\
 \label{s3eqn10}
 Y(a,z)b = \sum_{n\in\mathbb{Z}} a_{(n)}(b)z^{-n-1},\qquad a_{(n)}(b)=0\qquad \text{for} \ n\gg 0.
\end{gather}
\end{itemize}
The following axioms are required:
\begin{enumerate}\itemsep=0pt
	\item[$(1)$] \emph{Vacuum \& creation:} $Y(a,z)\mathbbm{1}$ is holomorphic for all $a\in V$ and
	\begin{align}
	 Y(a,z)\mathbbm{1} |_{z=0} = a,\qquad
 Y(\mathbbm{1},z) = \id_V.\label{s3eqn4}
	\end{align}
	\item[$(2)$] {\it $F$-translation covariance:} for all $a, b \in V$, we have
	\begin{align}
	 Y(\mathcal{S}(w)(a),z)b = i_{z,w} Y(a, F(z,w))b,\qquad
	 \mathcal{S}(z)\mathbbm{1} = \mathbbm{1}.\label{s3eqn5}
	\end{align}
	In \eqref{s3eqn5}, we have used the substitution \eqref{substitution2}. Moreover,
	$\mathcal{S}(z)\circ\mathcal{S}(w) =\mathcal{S}(F(z,w))$,
	$\mathcal{S}(0)=\id_V$.
	\item[$(3)$] \emph{Weak $F$-associativity:} for all $a,b,c \in V$, there exists $N \geqslant 0$ with
 \begin{equation}
 \label{s3eqn7}
 F(z,w)^N Y(Y(a,z)b,w)c = F(z,w)^N i_{z,w} Y(a, F(z,w)) Y(b,w) c,
 \end{equation}
 using the substitution $v\mapsto F(z,w)$ from \eqref{substitution1} on the right-hand side.
 \item[$(4)$] \emph{Skew symmetry:}
 \begin{equation}
 \label{s3eqn8}
 Y(a,z)b =\mathcal{S}(z) \circ Y(b, \iota(z))a.
 \end{equation}
\end{enumerate}
\end{Definition}

Putting $b=\mathbbm{1}$ into \eqref{s3eqn8} and using \eqref{s3eqn4} gives
\begin{equation}
\label{s3YD}
 Y(a,z)\mathbbm{1}=\mathcal{S}(z)a.
\end{equation}

\begin{Remark}
The axioms are slightly redundant. For example, \eqref{s3eqn5} implies \eqref{s3eqn7} for $b=\mathbbm{1}$. On the other hand, we cannot deduce \eqref{s3eqn5} from \eqref{s3eqn7} unless we know a priori that both sides of \eqref{s3eqn5} are meromorphic in $z$, $w$. Hence \eqref{s3eqn5} encodes a `meromorphicity' property.
\end{Remark}

\begin{Remark}
For an ordinary vertex algebra, one usually works with the translation operator~$D$, see~\cite{LL}. The shift operator then is $\mathcal{S}(z)=\exp(zD)$.
\end{Remark}

\begin{Remark}\label{ConjecturalRemark}
Huang \cite{Huang} has given a geometric interpretation of ordinary vertex algebras in terms of punctured Riemann spheres. For elliptic formal group laws, are vertex $F$-algebras analogously interpreted over an elliptic curve?

The author thanks an anonymous referee for pointing out that, in Huang's geometric interpretation of ordinary vertex algebras, conformal invariance plays a crucial role. In the genus one setting, an invariant construction requires not only the vertex algebra itself but also its modules and intertwining operators, whose matrix coefficients need not be meromorphic. Full modular invariance is presently established only in the $C_2$-cofinite (rational) case due to Huang \cite{HuangModular}, and more recently in certain $C_1$-cofinite settings in work of Creutzig--McRae--Yang \cite{CMY}. It therefore remains unclear whether a theory of vertex $F$-algebras over an elliptic curve could satisfy such invariance properties, although a formulation with weaker invariance (topological or birational) might still be meaningful.
\end{Remark}

\subsection{Meromorphicity}
\label{s42}

The meromorphicity assumption in \eqref{s3eqn10} can be combined with other axioms to prove the meromorphicity of various operator products and compositions. These, in turn, imply further non-trivial axioms. Here is a simple example of this principle.

\begin{Proposition}
For all $a,b \in V$ in a vertex $F$-algebra,
\begin{align}
\label{s3eqn9}
 i_{z,w}Y(a,F(z,w))\circ\mathcal{S}(w)b&=\mathcal{S}(w)\circ Y(a,z)b.
\end{align}
\end{Proposition}

\begin{proof}
 Putting $c=\mathbbm{1}$ into \eqref{s3eqn7} and using \eqref{s3YD} shows
 \begin{equation}
 \label{Faslkf}
 F(z,w)^N \bigl(\mathcal{S}(w)\circ Y(a,z)b \bigr) = F(z,w)^N\bigl(i_{z,w}Y(a,F(z,w))\circ\mathcal{S}(w)b\bigr)
 \end{equation}
 for some $N\geqslant0$. Since $\mathcal{S}(w)\circ Y(a,z)b \in V(\!( z,w)\!)$ by \eqref{F-shift} and \eqref{s3eqn10} and $Y(a,F(z,w))\mathcal{S}(w)b \in V(\!( z,w)\!)$ by \eqref{s3eqn5} and \eqref{s3eqn10}, we can embed equation \eqref{Faslkf} into $V(\!( z)\!)(\!( w)\!)$, where $F(z,w)$ is invertible. Equation \eqref{s3eqn9} follows.
\end{proof}

Here is the fundamental meromorphicity property of vertex $F$-algebras.

\begin{Proposition}
\label{prop:pabc}
For each $a,b,c\in V$, both sides of \eqref{s3eqn7} are a Laurent series $p_{a,b,c}(z,w)\in V(\!( z,w)\!)$. The formal fraction
\begin{equation}
\label{fABC}
 f_{a,b,c}(z,w)=\frac{p_{a,b,c}(z,w)}{F(z,w)^N}\in V(\!( z,w)\!)\big[F(z,w)^{-1}\big]
\end{equation}
is independent of $N$. Under the homomorphisms $i_{w,z}$ and $i_{z,w}$ from $V(\!( z,w)\!)\big[F(z,w)^{-1}\big]$ into $V(\!( w)\!)(\!( z)\!)$ and $V(\!( w)\!)(\!( z)\!)$, respectively, we have
\begin{align}
 i_{w,z} f_{a,b,c}(z,w)&=Y(Y(a,z)b,w)c,\label{s3eqn12}\\
 i_{z,w} f_{a,b,c}(z,w)&=i_{z,w}Y(a,F(z,w))Y(b,w)c.\label{s3eqn13}
\end{align}
\end{Proposition}

\begin{proof}
Using \eqref{s3eqn10}, we find that
\[
 Y(Y(a,z)b,w)c = \sum_{m,n\in\mathbb{Z}} \bigl(a_{(n)}(b)\bigr)_{(m)}(c) z^{-n-1}w^{-m-1}
\]
contains only finitely many negative powers of $z$. Moreover, for a fixed power of $z$, there are only finitely many negative powers of $w$. Hence the left-hand side of \eqref{s3eqn7} belongs to $V(\!( w)\!)(\!( z)\!)$. By definition of the substitution \eqref{substitution1}, the right-hand side of \eqref{s3eqn7} belongs to $V(\!( z)\!)(\!( w)\!)$. Hence both sides of \eqref{s3eqn7} are equal to a common series $p_{a,b,c}(z,w)\in V(\!( z,w)\!)$ in the intersection. If we embed $p_{a,b,c}(z,w)$ into $V(\!( w)\!)(\!( z)\!)$, we get $F(z,w)^N Y(Y(a,z)b,w)c$ and the series $F(z,w)$ is invertible in $V(\!( w)\!)(\!( z)\!)$, so we may rearrange and get \eqref{s3eqn12}. Similarly, we may embed into~${V(\!( z)\!)(\!( w)\!)}$ and rearrange to get \eqref{s3eqn13}.
\end{proof}

\begin{Proposition}[weak commutativity]
 For all $a,b,c\in V$ in a vertex $F$-algebra there exists~${M\geqslant0}$ with
\begin{equation}\label{s3eqn11}
 F(z,\iota w)^M Y(a,z)Y(b,w)c=F(z,\iota w)^M Y(b,w)Y(a,z)c.
\end{equation}
\end{Proposition}

We can replace $F(z,\iota w)^M$ by $(z-w)^M$ in \eqref{s3eqn11} as these differ by a unit \eqref{differByUnit}. Hence our vertex $F$-algebras satisfy also the axioms of Li~\cite{Li}.

\begin{proof}
 Combining skew symmetry with $F$-translation covariance gives
 \[
 Y(Y(a,z)b,w)c
 =i_{w,z} Y(Y(b,\iota z)a,F(z,w))c.
 \]
 By \eqref{s3eqn12}, we can reexpress this identity as
 \begin{equation}
 \label{wcc-eqn}
 i_{z,w}f_{a,b,c}(z,w)= i_{z,w}f_{b,a,c}(\iota z,F(z,w))
 \end{equation}
 in the ring $V(\!( z)\!)(\!( w)\!)$. Perform the substitution $z\to F(v,\iota w)$ from \eqref{substitution3} to get
 \begin{equation}
 \label{ProofWC1}
 i_{v,w}f_{a,b,c}(F(v,\iota w),w)=i_{v,w}f_{b,a,c}(F(w,\iota v),v).
 \end{equation}
 The same substitution applied to \eqref{s3eqn13} leads to
 \begin{equation}
 \label{ProofWC2}
 i_{v,w}f_{a,b,c}(F(v,\iota w),w)=Y(a,v)Y(b,w)c
 \end{equation}
 and exchanging $v\leftrightarrow w$, $a\leftrightarrow b$ gives
 \begin{equation}
 \label{ProofWC3}
 i_{w,v}f_{b,a,c}(F(w,\iota v),v)=Y(b,w)Y(a,v)c.
 \end{equation}
 Since $f_{b,a,c}(F(w,\iota v),v)=w^{-N}p_{b,a,c}(F(w,\iota v),v)$ by \eqref{fABC} and $p_{b,a,c}(u,v)$ contains only finitely many negative powers of $u$, there exists $M\geqslant0$ such that $F(w,\iota v)^M f_{b,a,c}(F(w,\iota v),v)$ lies in~${R(\!( v,w)\!)}$. From \eqref{swap-expansion} we thus have
 \[
 F(w,\iota v)^M i_{v,w}f_{b,a,c}(F(w,\iota v),v) = F(w,\iota v)^M i_{w,v}f_{b,a,c}(F(w,\iota v),v),
 \]
 which, combined with \eqref{ProofWC1}--\eqref{ProofWC3}, implies the result.
\end{proof}

\begin{Proposition}[Jacobi identity]
For all $a$, $b$ in a vertex $F$-algebra,
\begin{gather*}
i_{z_1,z_0}z_2^{-1}\delta_F\left(\frac{F(z_1,\iota z_0)}{z_2}\right)Y(Y(a,z_0)b,z_2)\\
\qquad=i_{z_1,z_2}z_0^{-1}\delta_F\left(\frac{F(z_1,\iota z_2)}{z_0}\right)Y(a,z_1)Y(b,z_2)\\
\phantom{\qquad=}{}-i_{z_2,z_1}z_0^{-1}\delta_F\left(\frac{F(z_1,\iota z_2)}{z_0}\right)Y(b,z_2)Y(a,z_1).
\end{gather*}
\end{Proposition}

\begin{proof}
Let $c\in V$. In the notation of \eqref{fABC}, set
\[
 \phi_{a,b,c}(z_0,z_1,z_2)=\frac{p_{a,b,c}(z_0,z_2)}{z_1^N}.
\]
Thus $\phi_{a,b,c}(z_0,z_1,z_2)=f_{a,b,c}(z_0,z_2)$ if $F(z_0,z_2)=z_1$. We have
\begin{gather}
i_{z_1,z_2}z_0^{-1}\delta_F\left(\frac{F(z_1,\iota z_2)}{z_0}\right)\phi_{a,b,c}(z_0,z_1,z_2)\nonumber\\
 \qquad\overset{\mathclap{\eqref{DiagonalSupport}}}{ = } i_{z_1,z_2}z_0^{-1}\delta_F\left(\frac{F(z_1,\iota z_2)}{z_0}\right) f_{a,b,c}(F(z_1,\iota z_2),z_2)\nonumber\\
\qquad \overset{\mathclap{\eqref{ProofWC2}}}{ = } i_{z_1,z_2}z_0^{-1}\delta_F\left(\frac{F(z_1,\iota z_2)}{z_0}\right)Y(a,z_1)Y(b,z_2).\label{ProofJacobi1}
\end{gather}
In the same way, \eqref{ProofWC3} implies
\begin{gather*}
 i_{z_2,z_1}z_0^{-1}\delta_F\left(\frac{F(z_1,\iota z_2)}{z_0}\right)\phi_{a,b,c}(z_0,z_1,z_2)\\
\qquad =i_{z_2,z_1}z_0^{-1}\delta_F\left(\frac{F(z_1,\iota z_2)}{z_0}\right) Y(b,z_2)Y(a,z_1).
\end{gather*}
The substitutions $z\to z_0$ and $w^{\pm1}\to i_{z_1,z_0}F(z_1,\iota z_0)^{\pm1}$ in \eqref{s3eqn12} give
\[ 
 i_{z_1,z_0} f_{a,b,c}(z_0,F(z_1,\iota z_0))=i_{z_1,z_0}Y(Y(a,z_0)b,F(z_1,\iota z_0))c.
\]
Therefore,
\begin{gather}
 i_{z_1,z_0}z_2^{-1}\delta_F\left(\frac{F(z_1,\iota z_0)}{z_2}\right)\phi_{a,b,c}(z_0,z_1,z_2)\nonumber\\
\qquad \overset{\mathclap{\eqref{DiagonalSupport}}}{ = }i_{z_1,z_0}z_2^{-1}\delta_F\left(\frac{F(z_1,\iota z_0)}{z_2}\right)f_{a,b,c}(z_0,F(z_1,\iota z_0))\nonumber\\
 \qquad= i_{z_1,z_0}z_2^{-1}\delta_F\left(\frac{F(z_1,\iota z_0)}{z_2}\right)Y(Y(a,z_0)b,F(z_1,\iota z_0))c.\label{ProofJacobi3}
\end{gather}
Now put \eqref{ProofJacobi1}--\eqref{ProofJacobi3} into the Jacobi identity of Proposition~\ref{PropJacobiDelta}.
\end{proof}

\subsection{Construction of Lie bracket}
\label{s43}

\begin{Theorem}
The formula
\begin{equation}
\label{s4eqn1}
[a,b] = \Res_{z=0}^F Y(a,z)b {\rm d}z
\end{equation}
defines a Lie bracket on the quotient $V/\sum_{n\geqslant1}\mathcal{S}^{(n)}(V)$.
\end{Theorem}

\begin{proof}
We first show that \eqref{s4eqn1} descends to the quotient by proving
\begin{equation}
\label{s4eqn2}
\Res^F_{z=0} Y\bigl(\mathcal{S}^{(m)}(a),z\bigr)b {\rm d}z=\frac{1}{m!}\left.\frac{d^m}{{\rm d}w^m}\right|_{w=0} \Res_{z=0}^F Y(\mathcal{S}(w)(a),z)b {\rm d}z
=0
\end{equation}
for $m\geqslant1$. This will also prove
\[
 \Res_{z=0}\biggl(\frac{Y(a,z)\mathcal{S}^{(m)}(b)}{F^{1,0}(0,z)}\biggr){\rm d}z\in\sum_{n\geqslant1}\mathcal{S}^{(n)}(V),
\]
since by \eqref{s3eqn5}, \eqref{s3eqn9} we have
\begin{align*}
Y(a,z)\circ\mathcal{S}(w)&=\mathcal{S}(w)\circ Y(a,F(z,\iota(w)))=\mathcal{S}(w)\circ Y(\mathcal{S}(\iota(w))a,z)\\
&=\sum_{k,\ell\geqslant0} \mathcal{S}^{(k)}\circ Y\big(\mathcal{S}^{(\ell)}(a),z\big)w^k\iota(w)^\ell,
\end{align*}
which expresses $Y(a,z)\mathcal{S}^{(m)}(b)$ for $m\geqslant 1$ as a sum of terms with $k\geqslant1$ or with $\ell\geqslant1$. The terms with $k\geqslant1$ belong to $\sum_{n\geqslant1}\mathcal{S}^{(n)}(V)$, while the terms with $\ell\geqslant1$ are covered by \eqref{s4eqn2}.

We show \eqref{s4eqn2} by applying Theorem~\ref{s2thm1}. Since
\begin{align*}
 \Res_{z=0}^F Y(\mathcal{S}(w)a,z)b {\rm d}z&=\Res_{z=0}^F Y(a,F(z,w))b {\rm d}z=\sum_{n\geqslant0}\bigl(\Res_{z=0}^F\mathcal{S}_n Y(a,z)b\bigr) w^n {\rm d}z\\
 &=\Res_{z=0}^F Y(a,z)b {\rm d}z,
\end{align*}
the series is constant in $w$, as required.

Clearly the bracket is bilinear. We prove the Jacobi identity. Using the notation of Proposition~\ref{prop:pabc}, set \smash{$f_{a,b,c}(z,v,w)=\frac{p_{a,b,c}(z,w)}{v^N}$}. Then
\begin{align*}
 i_{v,w}[f_{a,b,c}(z,v,w)|_{z\to F(v,\iota(w))}]&=Y(a,v)Y(b,w)c,\\
 i_{w,v}[f_{a,b,c}(z,v,w)|_{z\to F(v,\iota(w))}]&=Y(b,w)Y(a,v)c.
\end{align*}
On the other hand, putting $w=F(v,\iota(z))$ into \eqref{wcc-eqn} implies the first step in
\[
 i_{v,z}[f_{a,b,c}(z,v,w)|_{w\to F(v,\iota(z))}]=i_{v,z}\frac{p_{b,a,c}(\iota(z),v)}{F(\iota(z),v)^N}
 \overset{\eqref{s3eqn12}}{=} Y(Y(b,\iota z)a,v)c.
\]
Putting these calculations into the formula \eqref{s2eqn12} for iterated $F$-residues with $(z_0,z_1,z_2)\to(z,v,w)$ and using \eqref{residue-inversion} yields a version of the Jacobi identity
\[
 [a,[b,c]]-[b,[a,c]]=-[[b,a],c].
\]
Assuming a trivial center (meaning that $[d,c]=0$ for all $c$ implies $d=0$), this Jacobi identity implies also the skew symmetry of the Lie bracket. In general, \eqref{F-shift} and \eqref{s3eqn8} imply
\[
 Y(a,z)b + \sum_{n\geqslant1}\mathcal{S}^{(n)}(V) = Y(b,\iota(z))a + \sum_{n\geqslant1}\mathcal{S}^{(n)}(V)
\]
and then the skew symmetry follows from \eqref{residue-inversion}.
\end{proof}

\begin{Remark}
For ordinary homology, the moduli space of the quiver given by the Dynkin diagram of a Lie algebra recovers the original Lie algebra. What happens in the case of elliptic homology?
\end{Remark}

\subsection[Example: Heisenberg vertex F-algebra]{Example: Heisenberg vertex $\boldsymbol{F}$-algebra}
\label{s44}

The simplest new examples are generated by a single vertex operator. These are the Heisenberg vertex $F$-algebras. Define a central extension
\[
 0\longrightarrow R\longrightarrow H_F\longrightarrow R(\!( t)\!)\longrightarrow0
\]
of the commutative Lie algebra $R(\!( t)\!)$ by the cocycle
\[
 c(f,g)=\Res^F_{z=0} \mathcal{S}_1^F(f)\cdot g=\Res_{z=0} f'g.
\]
Notice~\eqref{s2eqn7} here. Set $b_n=t^n$. Then
\begin{equation}
\label{Heisenberg_Commutation_Rules}
 [b_n,b_m]=n\delta_{n,-m},
\end{equation}
so we have obtained the ordinary \emph{Heisenberg Lie algebra} except for having replaced $\mathbb{C}$ by the ring $R$. Define the space of states
$
 V=R[b_{-1},b_{-2},\dots]$.
Write $\mathbbm{1}\in V$ for the constant series $1$, the \emph{vacuum}. Define the vertex operator
\[
 b(z)=\sum_{n\in\mathbb{Z}} b_nz^{-n-1},
\]
where
\[
 b_n\colon\ V\longrightarrow V,\qquad
 b_n
 =
 \begin{cases}
 b_n\cdot(-) &(n<0),\\
 n\dfrac{\partial}{\partial b_{-n}} & (n\geqslant0).
 \end{cases}
\]
These operators satisfy the commutation rules~\eqref{Heisenberg_Commutation_Rules}. We seek an $F$-translation operator
\[
 \mathcal{S}(z)=\sum_{n\geqslant0} \mathcal{S}^{(n)}z^n,\qquad \mathcal{S}^{(n)}\colon\ V\longrightarrow V,
\]
satisfying $\mathcal{S}(z)\mathbbm{1}=\mathbbm{1}$ and
$
 [\mathcal{S}(z), b(w)]=b(F(z,w))$.
These conditions actually determine $\mathcal{S}(z)$ completely. Written in components,
\[
 \big[\mathcal{S}^{(n)},b_m\big] = \text{coefficient of $w^m$ in }\bigl(\mathcal{S}^{(n)}_Fb\bigr)(w).
\]
This can be applied inductively to a state $b_{j_1}\cdots b_{j_k}$ until we reach $\mathcal{S}^{(n)}\mathbbm{1}=0$.

\subsection*{Acknowledgements}

The author would like to thank an anonymous referee for several suggestions, in particular for the referee's comments on Remark~\ref{ConjecturalRemark}.

\pdfbookmark[1]{References}{ref}
\LastPageEnding

\end{document}